\documentclass[a4paper,12pt,oneside,reqno]{amsart}

\usepackage[utf8]{inputenc}
\usepackage[T1]{fontenc}

\usepackage{hyperref}
\usepackage[headinclude,DIV=13]{typearea}
\areaset{15.1cm}{25.0cm}
\usepackage{amsfonts,amssymb,amsmath,amsthm,bbm}
\usepackage{graphicx, psfrag}

\usepackage{subfigure}

\newtheorem{theorem}{Theorem}[section]
\newtheorem{lemma}[theorem]{Lemma}
\newtheorem{proposition}[theorem]{Proposition}

\theoremstyle{remark}
\newtheorem{remark}{Remark}[section]

\newcommand{\defeq}{\mathrel{\mathop:}=}

\def\paragraph#1{\noindent \textbf{#1}}

\numberwithin{equation}{section}

\def\eee{\mathrm{e}}

\def\<{\langle}
\def\>{\rangle}

\def\R{{\mathbb{R}}}  
\def\N{{\mathbb{N}}}  
\def\P{{\mathbb{P}}}  
 
\def\Z{{\mathbb{Z}}}

\def\EB{{\mathbf{E}}}

\def \ba {\begin{array}}
\def \ea {\end{array}}

\newcommand{\be}{\begin{equation}}
\newcommand{\ee}{\end{equation}}

\newcommand{\bea}{\begin{eqnarray}}
\newcommand{\eea}{\end{eqnarray}}
\def\TH(#1){\label{#1}}\def\thv(#1){\ref{#1}}
\def\Eq(#1){\label{#1}}\def\eqv(#1){(\ref{#1})}

\def \1{\mathbbm{1}}

\def\eee{\hbox{\rm e}}

\DeclareMathOperator{\diag}{diag}
\DeclareMathOperator{\vek}{vec}

\newcommand{\LKilled}{L_{\mathrm{kill}}}
\newcommand{\sigmaPerturbed}{\widetilde{\sigma}}
\newcommand{\SigmaPerturbed}{\widetilde{\Sigma}}
\newcommand{\uPerturbed}{\widetilde{u}}
\newcommand{\UPerturbed}{\widetilde{U}}
\newcommand{\vPerturbed}{\widetilde{v}}
\newcommand{\VPerturbed}{\widetilde{V}}
\newcommand{\I}{\mathbbm{1}}
\usepackage[
    backend=biber,
    style=alphabetic,
    giveninits=true,
    sortlocale=auto,
    date=year,
    natbib=true,
    url=false,
    doi=true,
    isbn=false,
    eprint=false,
    maxnames=5,
    maxcitenames=50 
]{biblatex}

\renewbibmacro{in:}{
  \ifentrytype{article}{}{\printtext{\bibstring{in}\intitlepunct}}}
\usepackage{subfigure}
\usepackage{pdfpages}
\usepackage{graphicx} 

\begin{document}

\title{Meeting times of Markov chains via singular value decomposition}
\author[T.M.~van~Belle]{Thomas van Belle}
\address{T.M.~van~Belle\\
Universität Duisburg-Essen\\
Fakultät für Mathematik\\
Thea-Leymann-Str.~9\\
45127 Essen\\
Germany}
\email{thomas.vanbelle@uni-due.de}

\author[A.~Klimovsky]{Anton Klimovsky}
\address{A.~Klimovsky\\
Julius-Maximilians-Universität Würzburg\\
Institut für Mathematik\\
Lehrstuhl für Mathematik VIII: Angewandte Stochastik\\
Emil-Fischer-Str.~30\\
97074 Würzburg\\
Germany
}
\email{anton.klimovsky@mathematik.uni-wuerzburg.de}



\date{\today}

\begin{abstract}
We suggest a non-asymptotic matrix perturbation-theoretic approach to get sharp bounds on the expected meeting time of random walks on large (possibly random) graphs. We provide a formula for the expected meeting time in terms of the singular value decomposition of the diagonally killed generator of a pair of independent random walks, which we view as a perturbation of the generator. Employing a rank-one approximation of the diagonally killed generator, as the proof of concept, we work out sharp bounds on the expected meeting time of simple random walks on sufficiently dense Erdős–Rényi random graphs.
\end{abstract}

\thanks{
This work was partly supported by the Deutsche Forschungsgemeinschaft (DFG, German Research Foundation) -- Project-IDs 412848929; 443891315.
}

\maketitle

\section{Introduction}

Let $G_n = (V_n := [n], E_n)$, $n \in \N$ be a sequence of (possibly random) undirected graphs. Consider a pair of independent discrete time random walks on $G_n$: $(X^{(n)}_t, Y^{(n)}_t)_{t \in \Z_+}$ started at vertex $i \in V_n$ resp.~vertex $j \in V_n$. How long does it take them to meet?
The \textit{meeting time of two independent random walks} on $G_n$ is a stopping time defined as
\begin{align}
    \label{eq:meeting-time}
    \tau_{\mathrm{meet}}^{(n)}(i,j) = \inf \{ t \in \Z_{\geq 0} : X^{(n)}_t = Y^{(n)}_t, X^{(n)}_0 = i, Y^{(n)}_0 = j \}, \quad i, j \in V_n
\end{align}
(with the natural convention that $\tau_{\mathrm{meet}}^{(n)}(i,j) = \infty$ if the random walks never meet.) Of interest is also the \textit{expected meeting time from stationarity}
\begin{align}\label{eq:meeting-time-stationary}
    t_{\mathrm{meet}}^{\pi} = \EB_{(i,j) \sim \pi \otimes \pi} [\tau_{\mathrm{meet}}^{(n)}(i,j)] = \sum_{i,j \in V_n} \pi _i \pi_j \EB\left[\tau_{\mathrm{meet}}^{(n)} (i,j)\right],
\end{align}
where $\pi$ is the invariant distribution of the random walk transition matrix $P$ and $\EB$ is the expectation w.r.t.~the law of the random walks on $G_n$. Both \eqref{eq:meeting-time} and \eqref{eq:meeting-time-stationary} are key quantities characterizing the behavior of random walks on graphs, see, e.g., \cite{AldousFill1995BookDraft}.

\subsection*{Outline of the article.}  
\label{sec:outline}
In Sections~\ref{sec:meeting-time-via-svd}-\ref{sec:application-ER-random-walk}, we state our main results. In Section~\ref{sec:related-literature}, we point to related literature. In Section~\ref{sec:spectral-formula}, we give proof of the spectral formula for the expected meeting time, and we deduce the rank-$k$ approximation bounds. In Section \ref{sec:matrixperturbation}, we apply a result from matrix perturbation theory to investigate properties of the diagonally killed generator of a pair of independent Markov chains. In Section~\ref{sec:dense-random}, we analyze the spectral formula applied to sufficiently dense Erd\H{o}s-R\'enyi random graphs to deduce sharp non-asymptotic bounds on the expected meeting time from stationarity. In Section~\ref{sec:discussion}, we close with a discussion.

\subsection{Meeting time formula via singular value decomposition}
\label{sec:meeting-time-via-svd}

The first main result is a formula for the expected meeting times of two independent Markov chains both with transition matrix $P$ from any pair of initial states in terms of the singular values decomposition (SVD) of what one might view as the
\textit{``diagonally killed'' generator of a pair of independent Markov chains}
\begin{align}
    \LKilled := I_{n^2} - (P\otimes P)E \in [-1,1]^{n^2 \times n^2},
\end{align}
$I_{n^2} \in \{0, 1\}^{n^2 \times n^2}$ is the unit matrix and $E \in \{0, 1\}^{n^2 \times n^2} $ is a matrix with 
\begin{align}
E_{ij} &:= \mathbbm{1} \{i=j\} \mathbbm{1} \{i\in \mathcal{D}\}, \quad i,j \in [n^2],
\end{align}
where $\mathcal{D} = \{1,n+2,2n+3,\dots,n^2\}$ and $\mathbbm{1}$ denotes the indicator function.
\begin{remark}[Meaning of $E$]
Let $C \in \R^{n^2 \times n^2}$ be a matrix. Then, the multiplication of $C$ by $E$, e.g., from the right zeros out $n$ columns of the matrix $C$, which exactly correspond to the diagonal $\{ (i,i) \in V \times V : i \in V \}$ in the product state space. Thus, $E$ encodes the meeting event. Note that $(P\otimes P) E$ is therefore a sub-stochastic matrix.
\end{remark}
In order to better understand the Kronecker product structure of matrices, we define an ``array flattening'' mapping $f\colon [n]^2 \to [n^2]$ by 
\begin{equation}
[n]^2 \ni (k,\ell) \overset{f}{\mapsto} (k-1)n + \ell \in [n^2].
\end{equation}
We will slightly abuse the notation and write as an index $(k,\ell)$ instead of $f(k,\ell)$. We choose this notation since it allows us to concisely write
\begin{equation}
    \left(P\otimes P\right)_{(i,j), (k,\ell)} = P_{ik} P_{j\ell}.
\end{equation}
Furthermore, the matrix $E$ can be written as
\begin{align}
\left(E\right)_{(i,j),(k,\ell)} &:= \I \{(i,j) = (k,\ell)\} \I \{k\neq\ell\}, \quad i,j,k,\ell \in [n].
\end{align}
For a matrix $C \in \R^{n\times n}$, we denote by $\vek(C)\in\R^{n^2}$ the vector such that
\begin{equation}
    \vek(C)_{(i,j)} = C_{ij}.
\end{equation}
In what follows, we equip $\R^n$ with the Euclidean ($\ell^2$-norm) $\Vert \cdot \Vert$ which induces the $2$-$2$ operator norm on the space of matrices, which we also denote by $\Vert \cdot \Vert$. 
Now we are in position to state the first main technical result.
\begin{proposition}[Expected meeting times via SVD of the diagonally killed generator] \label{prop:SVD}
Let $P$ the transition matrix of an irreducible Markov chain. Let $\sigmaPerturbed_1, \sigmaPerturbed_2, \ldots, \sigmaPerturbed_{n^2}$ be
the singular values (with multiplicities) of $\LKilled$ ordered non-increasingly:
\begin{align}
\sigmaPerturbed_1 \geq \sigmaPerturbed_2 \geq \ldots \geq \sigmaPerturbed_{n^2} > 0.
\end{align}
Let $\uPerturbed_{i}$ and $\vPerturbed_{i}$ be the corresponding left and right singular vectors, $i \in [n^2]$. Then,
\begin{align}
\label{eq:expected-meeting-matrix-formula}
    \vek(\EB[\tau_{\mathrm{meet}}^{(n)}(i,j)]_{i,j \in [n]}) = E(I_{n^2} - (P \otimes P) E)^{-1} \mathbf{\underline{1}}_{n^2},
\end{align}
and in particular
\begin{align}\label{eq:expected-meeting-matrix}
t_{\mathrm{meet}}^{\pi} &= \left(\pi\otimes\pi\right)^t E(I_{n^2} - (P \otimes P) E)^{-1} \mathbf{\underline{1}}_{n^2} \\
&= -1 + \sum_{i = 1}^{n^2} \frac{1}{\sigmaPerturbed_{i}} (\pi \otimes \pi)^t \vPerturbed_{i} \uPerturbed_{i}^t \mathbf{\underline{1}}_{n^2}, 
\label{eq:expected-meeting-spectral}
\end{align}
where $\mathbf{\underline{1}}_{n^2}$ is the $n^2$-dimensional vector with $1$ at all coordinates, $\pi$ is the stationary distribution of $P$. 
\end{proposition}

The second main technical result is the following non-asymptotic rank-$k$ approximation bound for the expected meeting time from stationarity.

\begin{proposition}[Rank-$k$ approximation of the meeting time]\label{lem:rankapprox}
Let $k \in [n^2]$, and define the \textbf{rank-k approximation of the expected meeting time} as
    \begin{align}
    \label{eq:rank-k-approximation}
       \hat{\text{t}}_{\mathrm{meet}}^{\pi, (k)} := -1 
        +\sum_{i = n^2 - k + 1}^{n^2} \frac{1}{\sigmaPerturbed_{i}} (\pi\otimes\pi)^t \vPerturbed_{i} \uPerturbed_{i}^t \mathbf{\underline{1}}_{n^2}.
    \end{align}
    Then, the expected meeting time can be approximated by $\hat{\text{t}}_{\mathrm{meet}}^{\pi, (k)}$:
    \begin{equation}
        \label{eq:rank-k-error-bound}
        \left|\hat{\text{t}}_{\mathrm{meet}}^{\pi, (k)} - t^{\pi}_\mathrm{meet}\right| \leq \frac{n\Vert\pi\Vert^2}{\sigmaPerturbed_{n^2 - k}}.
    \end{equation}
\end{proposition}

\begin{remark}
In ``mean-field-like'' graphs, in which $\pi$ is close to the uniform distribution, we get $\Vert \pi \Vert^2 = \Theta\left(\frac{1}{n}\right)$. So the l.h.s.~of \eqref{eq:rank-k-error-bound} reduces to $\Theta\left(\frac{1}{\sigmaPerturbed_{n^2 - k}}\right)$.
\end{remark}

\subsection{A naive matrix perturbation theoretic approach to computations with diagonally killed generator}\label{sec:naiveperturbation}
To apply Propositions~\ref{prop:SVD} and \ref{lem:rankapprox}, it can be useful to view $\LKilled$ as a perturbation of the \textit{generator of a pair of independent random walks}
\begin{align}
\label{eq:generator-p-p}
    L :=  I_{n^2} - (P\otimes P)
\end{align}
to control the r.h.s. of \eqref{eq:expected-meeting-spectral}. We denote the singular values (with multiplicities) of $L$ by 
\begin{equation}
\label{eq:unperturbed-singular-values}
    \sigma_1 \geq \sigma_2\geq\dots\geq\sigma_{n^2-1} > \sigma_{n^2} = 0.
\end{equation}
The corresponding left and right singular vectors are denoted by $u_i$ and $v_i$ respectively. A naive application of the non-asymptotic second order perturbation for the least singular value \cite{stewart1984second} yields
\begin{equation}\label{eq:perturbationexpansion}
     \sigmaPerturbed_{n^2}^2 = \sigma_{n^2}^2\left(\LKilled\right) = \gamma_{11}^2 + O\left(\Vert (P\otimes P)(E - I) \Vert^3\right),
\end{equation}
where $\gamma_{11} = u_{n^2}^t (P\otimes P)(E - I) v_{n^2}$. Note that the vectors $u_{n^2}$ and $v_{n^2}$ are given by
\begin{equation}\label{eq:unperturbedvectors}
    u_{n^2} = \frac{\pi\otimes\pi}{\Vert\pi\otimes\pi\Vert} \quad \text{and} \quad v_{n^2} = \frac{\mathbf{\underline{1}}_{n^2}}{\Vert\mathbf{\underline{1}}_{n^2}\Vert}.
\end{equation}
The expansion \eqref{eq:perturbationexpansion} can further be simplified by noting that universally, i.e., for any graph, 
\begin{equation}\label{eq:gamma11computation}
    \gamma_{11} = \frac{(\pi\otimes\pi)^t(P\otimes P)(E - I)\mathbf{\underline{1}}_{n^2}}{\Vert\pi\otimes\pi\Vert\Vert\mathbf{\underline{1}}_{n^2}\Vert} = \frac{(\pi\otimes\pi)^t(E-I)\mathbf{\underline{1}}_{n^2}}{\Vert\pi\otimes\pi\Vert\Vert\mathbf{\underline{1}}_{n^2}\Vert} = -\frac{1}{n}.
\end{equation}
This naive application already fails even for the complete graph: a simple computation shows that the norm of the perturbation $\Vert\left(P\otimes P\right)(E-I)\Vert$ is of order $n^{-1/2}$. Even though \eqref{eq:perturbationexpansion} cannot be applied directly, it does show that the least singular value of $\LKilled$ is shifted by $\frac{1}{n}$ plus some graph-dependent terms. As we show below, see Section~\ref{sec:matrixperturbation}, it is possible to upgrade our understanding of the graph-dependent terms to make the method applicable to sufficiently dense Erdős–Rényi random graphs. This yields rather sharp non-asymptotic bounds on \eqref{eq:expected-meeting-spectral}, see Theorem~\ref{thm:meetingtime}. In particular, the approximation 
\begin{align}
\label{eq:meeting-via-gamma-approximation}
t_{\mathrm{meet}}^{\pi} \approx  \frac{1}{|\gamma_{11}|} = n
\end{align}
turns out to be valid with high probability.

\begin{remark}
\label{rem:meeting-random-regular-graph}
A word of caution: It is known that the meeting time is \textit{not} universally asymptotically equivalent to $n$ (i.e., with proportionality constant $1$), cf., e.g., the random regular graph \cite{chen2021precise} for another proportionality constant than $1$ (but the same order $\Theta(n)$). Moreover, the meeting time must not be $\Theta(n)$ at all. It can be, e.g., $\Theta(1)$ like it is on star graphs, or $\Theta(n^2)$ like it is on the cycle graph, cf.~\cite[Table~1]{kanade_coalescence_2018}.
\end{remark}

\subsection{Meeting times of random walks on sufficiently dense Erdős–Rényi random graphs}
\label{sec:application-ER-random-walk}

In this section, we report on a rigorous proof-of-concept implementation of the idea for sufficiently dense Erdős–Rényi random graphs. Our result is non-asymptotic and identifies the exact proportionality constant. This constitutes our third main result.

Recall that the Erdős–Rényi random graph $\mathcal{G}(n,p)$ where $p \in (0,1)$, $n \in \N$ is a random graph on $V := [n]$ such that the symmetric adjacency matrix $A = (a_{i, j})_{i,j \in [n]}$ has independent Bernoulli$(p)$ distributed entries $(a_{i,j})_{i < j}$ and $a_{i,i} = 0$, $i \in V$. In what follows, we denote by $\P$ the law of the Erd\H{o}s-R\'enyi random graph.

\begin{theorem}[Meeting time for sufficiently dense Erdős–Rényi graphs]\label{thm:meetingtime}
    Consider two independent simple random walks on Erdős–Rényi random graph $\mathcal{G}(n, p_n)$, where the edge probability $p_n = cn^{\beta - 1} \wedge 1$ with $c > 0$ arbitrary but fixed. Let $\beta > \frac{1}{2}$. Then, for all $\epsilon > 0$, there exist $N \in \N$ and constants $\nu_1, \nu_2, \theta > 0$ such that for all $n\geq N$,
\begin{multline}
    \label{eq:meetingtimeconvergence}
        \P\left(\bigg|\frac{1}{n}t_{\mathrm{meet}}^{\pi} - 1 \bigg| > \epsilon \right)
        \leq 2n\exp\left(- \frac{\nu_1^2 d}{3}\right) + 2\binom{n}{2}\exp\left(- \frac{\nu_2^2 d^2}{3n}\right)
     + \eee^{-\theta \left(\log n\right)^{2}},
\end{multline}
where $d := np_n$ is the mean degree.
In particular, the r.h.s.~\eqref{eq:meetingtimeconvergence} converges to zero, as $n \to \infty$.
\end{theorem}
See Section~\ref{sec:discussion} for a discussion including possible extensions.

\subsection{Related literature}
\label{sec:related-literature}
There is a sizeable literature on the meeting times of random walks, see, e.g., \cite{AldousFill1995BookDraft,aldous_interacting_2013,kanade_coalescence_2018}. Meeting times have applications to the long-run behavior of interacting particle systems such as voter model, see, e.g., \cite{aldous_interacting_2013, oliveira_mean_2013, Mossel2017opinion, MasudaPorterLambiotte2017, chen2021precise, HermonEtAl2022, fernleyOrtgiese2023voter}, and some related algorithms on networks, see, e.g., \cite{cooper2010multiple, kanade_coalescence_2018}.

Recently, there has been an increased interest in meeting times on random graphs, see, e.g., \cite{durrett2010some, chen2021precise, fernleyOrtgiese2023voter, avena2022discordant, avena2023meeting}. 
For a sufficiently dense Erdős–Rényi random graph, i.e., $\beta > 0.5$, to the best of our knowledge, the sharpest bound on $t_\mathrm{meet}$ can be derived from \cite[Theorem~B.1]{kanade_coalescence_2018}. Specifically, $t_\mathrm{meet} = \Theta(n)$, which gives no information about the proportionality constant. For $\beta \in (0, 0.5]$, the result gives $t_\mathrm{meet} = O(n t_\text{mix})$ w.h.p.

Spectral characteristics (eigenvalues and -vectors) of the random walk transition matrix $P$ are widely used to analyze the behavior of random walks on graphs (and more generally, Markov processes). For example, mixing times can be bounded using the spectral gap and the Perron-Frobenius eigenvector of $P$~\cite{LevinPeres2017Book}. Hitting times can also be exactly expressed in terms of the eigenvalues and -vectors of $P$~\cite{lovasz1993random}. The latter expression has been applied in \cite{lowe2014hitting,lowe2021central} to analyze hitting times for random walks on Erdős–Rényi random graphs in the connected regime. Metastability can also be characterized spectrally~\cite{bovier2016metastability}.

Singular values of Markov chain generators were (independently of this research) employed to define the notion of a spectral gap for nonreversible Markov chains \cite{chatterjee2023spectral}. 

A formula resembling \eqref{eq:expected-meeting-matrix-formula} appeared in \cite{george2018meeting} (However, no link to singular values \eqref{eq:expected-meeting-spectral} and matrix perturbation theory was made therein.)

Matrix perturbation theory has a relatively long history. Asymptotic expansions were developed in physics literature, e.g., \cite{schrodinger1926quantisierung}. It became a key topic in functional, numerical analyses and statistics, see, e.g., \cite{kato2013perturbation,StewartSun1990matrix}. Non-asymptotic matrix perturbation theory \cite{davisKahan1970rotation, wedin1972perturbation} especially for the modern random matrix models remains a current topic of research; see, e.g., \cite{BaikBenArousPeche2005,benaych2012singular,ORourkeVuWang2018random,vershynin2018high,chenEtAl2021spectralReview,abbe2020entrywise} with numerous contemporary applications, e.g., in statistics and machine learning.

\section{Meeting times: A matrix representation and a spectral formula}
\label{sec:spectral-formula}
In this section, we derive the spectral formula \eqref{eq:expected-meeting-spectral}. Furthermore, we prove the rank-$k$ approximation \eqref{eq:rank-k-error-bound} bound.
\begin{proof}[Proof of Proposition \ref{prop:SVD}]
	First, we follow \cite{george2018meeting} to write the pairwise meeting time in vector form. That is, consider two simple random walks $X_n$ and $Y_n$ both evolving according to $P$, and recall that $\tau_{\mathrm{meet}}^{(n)}(i,j)$ denotes the meeting time of the two walks when started from states $i$ and $j$ respectively. The pairwise meeting time can be found recursively through
	\begin{align}
	\tau_{\mathrm{meet}}^{(n)}(i,j) = \begin{cases}
		1 & \text{with prob. } \sum_{k\in[n]} P_{ik} P_{jk} \\
		\tau_{\mathrm{meet}}^{(n)}(k,\ell) + 1 & \text{with prob. } P_{ik} P_{j\ell}.
	\end{cases}
	\end{align}
	For $i\neq j$, let $M_{ij} = \EB\left(\tau_{\mathrm{meet}}^{(n)}(i,j)\right)$, then it can be seen that,
	\begin{equation}
    \begin{aligned}
		M_{ij} &= 1 + \sum_{k\in[n]} \sum_{\ell \in [n]\setminus\{k\}} P_{ik} P_{j\ell}  M_{k\ell} \\
		&= 1 + \sum_{k\in[n]} P_{ik} \sum_{\ell \in [n]} P_{j\ell}  M_{k\ell} - \sum_{k\in[n]} P_{ik} P_{jk} M_{kk}. 
	\end{aligned}
    \end{equation}
	Written in matrix form, we obtain,
	\begin{equation}\label{eq:meetingtimematrix}
	    M = \mathbf{\underline{1}}_n\mathbf{\underline{1}}_n^t + P\left(M-M_d\right)P^t,
	\end{equation}
	where $M_d$ is a diagonal matrix with the same diagonal elements as $M$. Rewriting in vector form, yields
	\begin{equation}
    \begin{aligned}
		\vek(M) &= \mathbf{\underline{1}}_{n^2} + \left(P\otimes P\right)\left(\vek(M) - \vek(M_d)\right) \\ 
		&= \mathbf{\underline{1}}_{n^2} + \left(P\otimes P\right)\left(I_{n^2} - \vek(I_n)\right)\vek(M) \\
		&= \mathbf{\underline{1}}_{n^2} + \left(P\otimes P\right)E\vek(M).
	\end{aligned}
    \end{equation}
	By rearranging terms it can be seen that, 
	\begin{align}
	\vek\left(M\right) = \left(I_{n^2} - \left(P\otimes P\right)E\right)^{-1} \mathbf{\underline{1}}_{n^2}.
	\end{align}
 Note that $\EB[\tau_{\mathrm{meet}}^{(n)}(i,i)] = 0$ for all $i\in [n]$, so that
 \begin{equation}
     \EB[\tau_{\mathrm{meet}}^{(n)}(i,j)] = M_{ij}\mathbbm{1}\{i\neq j\} = (E\vek(M))_{(i,j)}
 \end{equation}
 hence Equation \eqref{eq:expected-meeting-matrix-formula} follows. We observe that
\begin{align}
t_\mathrm{meet}^\pi = \sum_{i = 1}^n \sum_{\substack{j=1 \\ j \neq i}}^n \pi_i\pi_j M_{ij} = (\pi\otimes\pi)^tE\vek(M),
\end{align}
so that \eqref{eq:expected-meeting-matrix} easily follows. The expected meeting time can also be written as
\begin{align}\label{eq:meetingtimesplit}
 t_\mathrm{meet}^\pi = \sum_{i = 1}^n \sum_{\substack{j=1 \\ j \neq i}}^n \pi_i\pi_j M_{ij} = \sum_{i = 1}^n \sum_{j=1 }^n \pi_i\pi_j M_{ij} - \sum_{i = 1}^n \pi_i^2M_{ii}.
\end{align}
First, we compute the second term. Using \eqref{eq:meetingtimematrix} it can be seen that
\begin{align}
	\pi^tM\pi = \pi^t\mathrm{1}_n\mathrm{1}_n^t\pi + \pi^tP\left(M-M_d\right)P^t\pi = 1 + \pi^t\left(M-M_d\right)\pi.
 \end{align}
So that,
 \begin{align}
 \sum_{i = 1}^n \pi_i^2M_{ii} = \pi^t M_d \pi = 1.
 \end{align}
The first term of \eqref{eq:meetingtimesplit} can be written as
\begin{align}
	\sum_{i = 1}^n \sum_{j=1 }^n \pi_i\pi_j M_{ij} = \left(\pi\otimes\pi\right)^t\left(I_{n^2} - \left(P\otimes P\right)E\right)^{-1}\mathbf{\underline{1}}_{n^2}.
 \end{align}
To find the inverse in the equation above, we use the singular value decomposition of $\LKilled$, i.e., $I - (P\otimes P)E = \UPerturbed \SigmaPerturbed \VPerturbed^t$. Then,
\begin{equation}
    (I - (P\otimes P)E)^{-1} = \VPerturbed\SigmaPerturbed^{-1} \UPerturbed^t = \sum_{i=1}^{n^2} \frac{1}{\sigmaPerturbed_{i}}  \vPerturbed_{i} \uPerturbed_{i}^t.
\end{equation}
Equation \eqref{eq:expected-meeting-spectral} follows by plugging these results into \eqref{eq:meetingtimesplit}.
\end{proof}
\begin{proof}[Proof of Proposition~\ref{lem:rankapprox}]
From Proposition~\ref{prop:SVD}, it follows that
\begin{equation}
    t_{\mathrm{meet}}^{\pi} = -1 + \sum_{i=1}^{n^2} \frac{1}{\sigmaPerturbed_{i}} (\pi\otimes\pi)^t\vPerturbed_{i} \uPerturbed_{i}^t\mathbf{\underline{1}}_{n^2}.
\end{equation}
    The result then follows from the observation that
    \begin{equation}
        \begin{aligned}
             \left|\hat{\text{t}}_{\text{meet}}^\pi - t_{\mathrm{meet}}^{\pi}\right| &= \bigg| (\pi\otimes\pi)^t\left( \sum_{i = n^2 -k + 1}^{n^2}\frac{ \vPerturbed_{i} \uPerturbed_{i}^t}{\sigmaPerturbed_{i}} - \sum_{i=1}^{n^2} \frac{\vPerturbed_{i} \uPerturbed_{i}^t}{\sigmaPerturbed_{i}}  \right) \mathbf{\underline{1}}_{n^2}\bigg| \\
        &\leq \Vert\pi\otimes\pi\Vert\bigg\Vert \sum_{i = 1}^{n^2 - k} \frac{\vPerturbed_{i} \uPerturbed_{i}^t}{\sigmaPerturbed_{i}}  \bigg\Vert\Vert\mathbf{\underline{1}}_{n^2}\Vert \\
        &= \frac{1}{\sigmaPerturbed_{n^2-k}} n\Vert\pi\Vert^2
        ,
        \end{aligned}
    \end{equation}
    where the last equality follows from the Eckhart-Young(-Mirsky) Theorem~\cite{eckart1936approximation} and the observation that
    \begin{equation}
        \Vert\pi\otimes\pi\Vert^2 = \sum_{i = 1}^{n^2} (\pi\otimes\pi)_i^2 = \sum_{(i,j) = (1,1)}^{(n,n)} (\pi\otimes\pi)_{(i,j)}^2 = \sum_{i = 1}^{n}\sum_{j = 1}^{n} \pi_i^2\pi_j^2 = \Vert\pi\Vert^4.
    \end{equation}
\end{proof}

\section{Matrix Perturbation Theory}\label{sec:matrixperturbation}
To effectively utilize Proposition~\ref{lem:rankapprox}, we need to obtain sufficiently sharp estimates for the $k+1$ least singular values of the diagonally killed generator $\LKilled$. We use a result from non-asymptotic matrix perturbation theory \cite[Theorems 4.11 and 4.12]{stewart1973error} to obtain useful bounds on the least singular value.  For reader's convenience, we provide the statement of this result in the Appendix, cf.~Theorem~
\ref{thm:Stewart-perturbation} Recall the definitions of $v_{n^2}$ and $u_{n^2}$ in \eqref{eq:unperturbedvectors} and $\sigma_{n^2}$ in \eqref{eq:unperturbed-singular-values}, and denote the singular value decomposition of the unperturbed matrix $L$ in block form by
\begin{equation}
   L =  I_{n^2} - (P\otimes P)  =: \left[u_{n^2} \quad U_2\right]^t\left[\begin{matrix}
        \sigma_{n^2} & 0 \\
        0 & \Sigma_2
    \end{matrix}    
    \right] \left[v_{n^2} \quad V_2\right],
\end{equation}
where $U_2$ has the columns $\left(u_i\right)_{i<n^2}$ and $V_2$ has the columns $\left(v_i\right)_{i<n^2}$, and $\Sigma_2 := \diag (\sigma_i)_{i<n^2}$ is a diagonal matrix with the remaining singular values.
We apply Theorem~\ref{thm:Stewart-perturbation} with
\begin{equation}
    \begin{aligned}
        B &= L^tL \\
        \Delta &= \LKilled^t \LKilled - L^tL \\
        X &= \left[v_{n^2} \quad V_2\right].
    \end{aligned}
\end{equation}
Note that in this application the block matrices defined in $\eqref{eq:blocksA}$ are given by $B_{11} = 0$, $B_{22} = \Sigma_2^2$, $B_{12} = 0$. We define the parameters $\gamma_{11}$, $g_{12}$, $g_{21}$, and $G_{22}$ through
\begin{equation}
    \left[u_{n^2}\quad U_2\right]^t(P\otimes P)(E-I)\left[v_{n^2}\quad V_2\right] = \left[\begin{matrix}
        \gamma_{11} & g_{12}^t \\
        g_{21} & G_{22}
    \end{matrix}    
    \right].
\end{equation}
By observing that
\begin{equation}
    \left[u_{n^2}\quad U_2\right]^t(I - (P\otimes P)E)\left[v_{n^2}\quad V_2\right] = \left[\begin{matrix}
        \gamma_{11} & g_{12}^t \\
        g_{21} & \Sigma_2 + G_{22}
    \end{matrix}    
    \right],
\end{equation}
and that $B_{11} = \sigma_{n^2} = 0$, it can be seen that
\begin{equation}
    X^t(B + \Delta)X = \left[\begin{matrix}
        \gamma_{11}^2 + g_{21}^tg_{21} & g_{21}^t\Sigma_2 + \gamma_{11}g_{12}^t + g_{21}^tG_{22} \\
        \Sigma_2g_{21} + \gamma_{11}g_{12} + G_{22}^tg_{21} & (\Sigma_2 + G_{22})^t(\Sigma_2 + G_{22}) + g_{12}g_{12}^t
    \end{matrix}    
    \right].
\end{equation}
Therefore, the blocks in \eqref{eq:blocksDelta} are given by,
\begin{equation}
    \begin{aligned}
        \Delta_{11} &= \gamma_{11}^2 + g_{21}^tg_{21}, \\
        \Delta_{12} &= g_{21}^t\Sigma_2 + \gamma_{11}g_{12}^t + g_{21}^tG_{22}, \\
        \Delta_{21} &= \Sigma_2g_{21} + \gamma_{11}g_{12} + G_{22}^tg_{21}, \\
        \Delta_{22} &= G_{22}^t\Sigma_2 + \Sigma_2G_{22} + G_{22}^tG_{22} + g_{12}g_{12}^t.
    \end{aligned}
\end{equation}
Let $B\in \R^{\ell\times\ell}$ and $C\in\R^{m\times m}$. Define the operator $T_{B,C}\colon \R^{m\times\ell} \to \R^{m\times\ell}$ through,
\begin{equation}
    T_{B,C}P = PB - CP.
\end{equation}
The separation of $B$ and $C$ is the number $\text{sep}(B,C)$ defined by
\begin{equation}
    \text{sep}(B,C) := \begin{cases}
        \Vert T_{B,C}^{-1}\Vert^{-1}, & \text{if } 0\notin\lambda(T),\\
        0, & \text{otherwise}.
    \end{cases}
\end{equation}
Let 
\begin{equation}
    \delta = \text{sep}(B_{11}, B_{22}) - \Vert\Delta_{11}\Vert - \Vert\Delta_{22}\Vert.
\end{equation}
Note that since $B_{11} = 0$, and $B_{22} = \Sigma_2^2$, we get
\begin{equation}
    \text{sep}(B_{11}, B_{22}) = \Vert\Sigma_2^{-2}\Vert^{-1}.
\end{equation}
If
\begin{equation}
    \frac{4}{\delta^2}\Vert\Delta_{21}\Vert\left(\Vert B_{12}\Vert + \Vert\Delta_{12}\Vert\right) \leq 1,
\end{equation}
then there exists a matrix $Q$ satisfying $\Vert Q\Vert\leq\frac{2}{\delta} \Vert\Delta_{12}\Vert$ such that
\begin{equation}\label{eq:sigmaperturbedexact}
    \sigmaPerturbed_{n^2} = \lambda_{\text{min}}\left(B + \Delta\right)  = \gamma_{11}^2 + g_{21}^tg_{21} + (g_{21}^t\Sigma_2 + \gamma_{11}g_{12}^t + g_{21}^tG_{22})Q.
\end{equation}
Furthermore, the columns of $X_1^{'} = (1 + \Vert Q\Vert^2)^{-\frac{1}{2}}(v_{n^2} + V_2Q)$ span an invariant subspace of $B + \Delta$. Note that \eqref{eq:sigmaperturbedexact}, gives us an exact expression for the least singular value of the diagonally killed generator $\LKilled$. As we have seen in \eqref{eq:gamma11computation}, $\gamma_{11} = -\frac{1}{n}$ for any transition matrix $P$. The remaining terms depend on the structure of the graph and therefore need to be bounded on a case by case basis. In Section~\ref{sec:dense-random}, we will bound the terms $g_{12}, g_{21}, G_{22}, \Sigma_2$, and $Q$ by their operator norms, for the specific case where the graph is an Erd\H{o}s-R\'enyi random graph. In order to obtain useful upper- and lower-bounds on the least singular value, we define 
\begin{equation}
\begin{aligned}
    \Tilde{\gamma}_{11}^2 &\defeq \Vert g_{21} \Vert^2 + \gamma_{11}^2 = v_{n^2}^t(E-I)^t(P\otimes P)^t U_2 U_2^t (P\otimes P)(E-I) v_{n^2} + \gamma_{11}^2 \\&= v_{n^2}^t (E-I)^t(P\otimes P)^t\left(I - u_{n^2}u_{n^2}^t\right) (P\otimes P)(E-I) v_{n^2} + \gamma_{11}^2 \\
    &= \frac{1}{n^2} \mathbf{\underline{1}}_{n^2} (E-I)^t (P\otimes P)^t(P\otimes P)(E-I) \mathbf{\underline{1}}_{n^2},
\end{aligned} 
\end{equation}
to see that
\begin{equation}\label{eq:uplowsigma}
\begin{aligned}
    \sigmaPerturbed_{(n,n)}^2 &\geq \Tilde{\gamma}_{11}^2 - 2\cdot\frac{\Vert\Delta_{12}\Vert^2}{\delta}, \\
    \sigmaPerturbed_{(n,n)}^2 &\leq \Tilde{\gamma}_{11}^2 + 2\cdot\frac{\Vert\Delta_{12}\Vert^2}{\delta}.
\end{aligned}
\end{equation}
In Section~\ref{sec:dense-random}, we provide an example of graphs in which the term $2\cdot\frac{\Vert\Delta_{12}\Vert^2}{\delta}$ is of lower order than $\Tilde{\gamma}_{11}^2$. If this latter condition is satisfied, we can estimate the least singular value and use Propositions~\ref{prop:SVD} and \ref{lem:rankapprox} to conclude that w.h.p.~the approximation in \eqref{eq:meeting-via-gamma-approximation} is valid.

\section{Application to dense random graphs}
\label{sec:dense-random}
In this section, we will provide an application of Proposition \ref{prop:SVD} to dense Erdős–Rényi graphs. This will result in a proof of Theorem \ref{thm:meetingtime}. 

We start by proving some concentration results. It will be useful to define the following quantities,
\begin{equation}\label{eq:R_1def}
    \begin{aligned}
        \epsilon_i \defeq \frac{\deg i}{d} - 1 \qquad  R_1 \defeq \max_{i\in V} |\epsilon_i|,
    \end{aligned}
\end{equation}
where $d := cn^\beta$ is the average degree for each vertex.
Note these definitions imply
\begin{equation}\label{eq:maxmindegree}
    \min_{i\in V} \deg i \geq d(1 - R_1) \qquad \max_{i\in V} \deg i \leq d(1 + R_1).
\end{equation}
Furthermore, for $k\neq\ell$ define
\begin{equation}\label{eq:R_2def}
    \epsilon_{(k,\ell)} \defeq \frac{n}{d^2}\left(\sum_i \I\{k\sim i\}\I\{\ell \sim i\}\right) - 1 \qquad  R_2 \defeq \max_{(k,\ell)\in V\times V, k\neq\ell} |\epsilon_{(k,\ell)}|.
\end{equation}
Finally, for $\nu_1, \nu_2 > 0$, we define the events
\begin{equation}
        F_{\nu_1,\nu_2} \defeq \left\{R_1 \leq \nu_1 \right\} \cap \left\{R_2 \leq \nu_2\right\} \quad \text{and} \quad F_{\nu_1}\defeq \left\{R_1 \leq \nu_1 \right\}.
\end{equation}
\begin{lemma}\label{lem:concentrationR}
    For any $\nu_1,\nu_2 > 0$,
    \begin{equation}\label{eq:F1compup}
        \P\left(F_{\nu_1}^c\right) \leq 2n\exp\left(- \frac{\nu_1^2 np}{3}\right),
    \end{equation}
    and
    \begin{equation}\label{eq:F12compup}
        \P\left(F_{\nu_1,\nu_2}^c\right) \leq 2n\exp\left(- \frac{\nu_1^2 np}{3}\right) + 2\binom{n}{2}\exp\left(- \frac{\nu_2^2 np^2}{3}\right).
    \end{equation}
\end{lemma}
\begin{remark}
    Note that the upper-bound in \eqref{eq:F1compup} goes to zero whenever $np$ grows faster than $\log n$, which holds for any $\beta > 0$. The upper-bound in \eqref{eq:F12compup} however, goes to zero whenever $np^2$ grows faster than $\log\binom{n}{2}$. Since $np^2 = \frac{1}{c^2} n^{1-2\beta}$, this is the case for $\beta > \frac{1}{2}$.
\end{remark}
\begin{proof}
We use the union bound to get that for any $\nu_1>0$,
\begin{equation}
    \begin{aligned}
\P\left(R_1 > \nu_1\right) = \P\left(\exists i\in V : \bigg|\frac{\deg_i}{np} - 1\bigg| > \nu_1 \right) \leq n\P\left(\bigg|\frac{\deg_1}{np} - 1\bigg| > \nu_1\right).
\end{aligned}
\end{equation}
The probability on the right-hand side can be bounded using a Chernoff bound, i.e.,
\begin{align}
    \P\left(\bigg|\frac{\deg_1}{np} - 1\bigg| > \nu_1\right) = \P\left(\big|\deg_1 - np\big| > np\nu_1\right) \leq 2\exp\left(-\frac{\nu_1^2 np}{3}\right).
\end{align}
    Equation \eqref{eq:F1compup} follows. Similarly, note that
    \begin{equation}
    \begin{aligned}
        \P\left(R_2 > \nu_2\right) &= \P\left(\exists k\neq\ell : \bigg|\frac{1}{np^2}\sum_i \I\{k\sim i\}\I\{\ell \sim i\} - 1\bigg| > \nu_2\right) \\ &\leq \binom{n}{2} \P\left(\bigg|\frac{1}{np^2}\sum_i \I\{1\sim i\}\I\{2 \sim i\} - 1\bigg| > \nu_2\right) \\
        &\leq 2\binom{n}{2} \exp\left(-\frac{\nu_2^2np^2}{3}\right).
    \end{aligned}
    \end{equation}
    The standard bound
    \begin{equation}
        \P\left(F_{\nu_1,\nu_2}^c\right) \leq \P\left(R_1 > \nu_1\right) + \P\left(R_2 > \nu_2\right)
    \end{equation}
    completes the proof of \eqref{eq:F12compup}.
\end{proof}
Another graph-dependent event we frequently use is the convergence of the second least singular value. Denote the adjacency matrix of $G = G_n$ by $A = A_n$. We denote by $F$ the event that the second least singular value of $\frac{A}{\sqrt{d}}$ is at most $8$:
\begin{equation}
    F := \left\{\sigma_2\left(\frac{A}{\sqrt{d}}\right) \leq 8\right\}.
\end{equation}
We need the following adaptation from \cite{erdHos2012spectral, erdos2013} as follows,
\begin{proposition}\label{prop:EKYYsingular}
    There exist a constant $\theta > 0$ and $N\in \mathbb{R}_+$, such that for all $n\geq N$,
    \begin{equation}
        \P\left(F^c\right) \leq \exp\left(-\theta \left(\log n\right)^{2}\right)
    \end{equation}
\end{proposition}
\begin{remark}
    In fact, \cite{erdHos2012spectral} shows that for $d \gg (\log n)^{2C\xi}$,
    \begin{equation}
        \sigma_2\left(\frac{A}{\sqrt{d}}\right) \rightarrow 2
    \end{equation}
    with high probability. For our purposes, it suffices that $\sigma_2\left(\frac{A}{\sqrt{d}}\right)$ is bounded. The constant $8$ is arbitrary and chosen for (subjective) aesthetic reasons. It could be replaced by any constant bigger than $2$.
\end{remark}
\begin{proof}
    By \cite[Equation (3.19)]{erdHos2012spectral}, there exist constants $\theta, C > 0$ such that
    \begin{equation}
        \sigma_2\left(\frac{A}{\sqrt{d}}\right) \leq \sigma_2\left(\frac{A}{\sqrt{np(1-p)}}\right) \leq 2 + \left(\log n\right)^{C\xi}\left(\frac{1}{d} + \frac{1}{n^{2/3}}\right)
    \end{equation}
    holds with probability at least $1 - \exp\left(-\theta \left(\log n\right)^\xi\right)$, where
    \begin{equation}
        1 + a_0 \leq \xi \leq A_0 \log\log n
    \end{equation}
    for constants $a_0 > 0$ and $A_0 \geq 10$. We pick $\xi = 2$ to see that
    \begin{equation}
        \sigma_2\left(\frac{A}{\sqrt{d}}\right) \leq 2 + \left(\log n\right)^{2C}\left(\frac{1}{d} + \frac{1}{n^{2/3}}\right)
    \end{equation}
    holds with probability at least $1 - \exp\left(-\theta \left(\log n\right)^2\right)$.
    Let $n$ be such that,
    \begin{equation}\label{eq:log-vs-sqrtn}
         \frac{\left(\log n\right)^{2C}}{\sqrt{n}} \leq 3c \quad \text{ and } \quad  \frac{\left(\log n\right)^{2C}}{n^{2/3}} \leq 3.
    \end{equation}
    Note that there exists $N\in\R_+$ such that \eqref{eq:log-vs-sqrtn} is satisfied for all $n\geq N$. Therefore, for all $n \geq N$,
    \begin{equation}
        \sigma_2\left(\frac{A}{\sqrt{d}}\right) \leq 8
    \end{equation}
    holds with probability at least $1 - \exp\left(-\theta \left(\log n\right)^2\right)$. 
\end{proof}
Another key quantity is the norm of the perturbation, i.e., $\Vert \left(P\otimes P\right)(E - I)\Vert$. It can be usefully bounded in terms of the previously defined quantities $R_1$ and $R_2$ as follows.
\begin{lemma}\label{lem:NormPerturbation}
    Let $\beta \in (0,1)$ then,
    \begin{equation}\label{eq:normperturbationupper}
        \Vert (P\otimes P)(E-I)\Vert^2 \leq \frac{(1+R_1)^2}{d^2(1-R_1)^4} + \frac{(1+R_2)^2}{n(1-R_1)^4} - \frac{(1+R_2)^2}{n^2(1-R_1)^4}.
    \end{equation}
    For $\beta > \frac{1}{2}$,
    \begin{equation}\label{eq:normperturbationlower}
        \Vert (P\otimes P)(E-I)\Vert^2 \geq \frac{(1-R_1)^2}{d^2(1+R_1)^4} + \frac{(1-R_2)^2}{n(1+R_1)^4} - \frac{(1-R_2)^2}{n^2(1+R_1)^4}
        .
    \end{equation}
\end{lemma}
\begin{remark}
    This lemma shows that for $\beta > \frac{1}{2}$ the norm of the perturbation behaves like
    \begin{equation}
        \Vert (P\otimes P)(E-I)\Vert^2 \sim \frac{1}{n}.
    \end{equation}
    For $\beta < \frac{1}{2}$, $R_2$ grows. Therefore, we expect that the upper-bound in \eqref{eq:normperturbationupper} is not sharp.
\end{remark}
\begin{proof}
        Define the matrix $S\in\R^{n\times n}$     \begin{equation}
        S_{i,k} = \left(\sum_{q} A_{qi}A_{qk}\right)^2 = \begin{cases}
            d^2(1 + \epsilon_i)^2 & \text{if } i=k \\
            \frac{d^4}{n^2}(1 + \epsilon_{(i,k)})^2 & \text{if } i\neq k,
        \end{cases}
    \end{equation}
    so that it contains the non-zero elements of $(E-I)^t(A\otimes A)^t(A\otimes A)(E-I)$. Then
    \begin{equation}\label{eq:uppernormS}
        \Vert (P\otimes P)(E-I)\Vert^2 \leq \frac{\Vert S\Vert}{d^4(1-R_1)^4} \text{ and } \Vert (P\otimes P)(E-I)\Vert^2 \geq \frac{\Vert S\Vert}{d^4(1+R_1)^4}.
    \end{equation}
    Note that by the Gershgorin Circle Theorem,
    \begin{equation}\label{eq:normSup}
        \Vert S\Vert \leq d^2\left(1 + R_1\right)^2 + \frac{d^4}{n}\left(1 + R_2\right)^2 - \frac{d^4}{n^2}\left(1 + R_2\right)^2.
    \end{equation}
    Equation \eqref{eq:normperturbationupper} follows from \eqref{eq:uppernormS} and \eqref{eq:normSup}. To prove the second statement \eqref{eq:normperturbationlower}, we note that,
        \begin{equation}\label{eq:lowernormS}
        \Vert (P\otimes P)(E-I)\Vert^2 \geq \frac{\Vert S\Vert}{d^4(1+R_1)^4}.
    \end{equation}
    Furthermore, 
    \begin{equation}\label{eq:normSlow}
        \begin{aligned}
            \Vert S\Vert = \max_{\Vert x\Vert = 1} x^tSx \geq \frac{1}{n} \sum_{(i,k)}  S_{i,k} \geq d^2\left(1-R_1\right)^2 + \left(1 - \frac{1}{n}\right)\frac{d^4}{n}\left(1-R_2\right)^2
            ,
        \end{aligned}
    \end{equation}
    where the first inequality follows from taking $x$ to be the uniform vector. By combining \eqref{eq:lowernormS} and \eqref{eq:normSlow}, we obtain \eqref{eq:normperturbationlower}. 
\end{proof}
Next, we show that some key quantities in the application of Theorem \ref{thm:Stewart-perturbation} are bounded. These bounds will provide the backbone of the proof of Theorem \ref{thm:meetingtime}.
\begin{lemma}\label{lem:normestimates}
    Let $\beta > \frac{1}{2}$. For every $\varepsilon_1 > 0$, there exists $N_1\in\R_+$ and constants $\nu_1, \nu_2 > 0$ such that, for all $n \geq N_1$,
    \begin{align}\label{eq:Sigmaup}
        &\P\left(\Vert\Sigma_2\Vert\leq 2 + \varepsilon_1\mid F_{\nu_1}\right) = 1,
\\
\label{eq:g12up}
        &\P\left(n\Vert g_{12}\Vert^2 \leq 1 +\varepsilon_1\mid F_{\nu_1}\right) = 1,
\\
\label{eq:invariantup}
        &\P\left(n\Vert \pi\Vert^2 \leq 1 +\varepsilon_1\mid F_{\nu_1}\right) = 1,
\\
\label{eq:G22up}
        &\P\left(\sqrt{n}\Vert G_{22}\Vert \leq 1 +\varepsilon_1\mid F_{\nu_1, \nu_2}\right) = 1,
\\
\label{eq:g21up}
        &\P\left(n^2\Vert g_{21}\Vert^2 \leq (1 +\varepsilon_1)^2 - 1\mid F_{\nu_1, \nu_2}\right) = 1,
\\
\label{eq:Sigmainv}
        &\P\left((1 - \varepsilon_1)^2 \leq\Vert \Sigma_2^{-2}\Vert^{-1} \leq (1 + \varepsilon_1)^2\mid F_{\nu_1} \cap F\right) = 1.
    \end{align}
\end{lemma}
\begin{proof}
Let $\varepsilon_1 > 0$, and define the $\nu_1$ and $\nu_2$ by
\begin{equation}\label{eq:nudefs}
\nu_1 = \nu_2 = \frac{\left(\varepsilon_1 + 1\right)^\frac{1}{4} - 1}{\left(\varepsilon_1 + 1\right)^\frac{1}{4} + 1}.
\end{equation}
and recall that conditioning on $F_{\nu_1,\nu_2}$ and $F$ gives the conditions that $R_1 \leq \nu_1$, $R_2 \leq \nu_2$, and $\sigma_2\left(A\right) \leq 8\sqrt{d}$. This choice for $\nu_1$ and $\nu_2$ is not optimal, but it ensures that $\nu_1$ and $\nu_2$ are positive and bounded by $1$. Furthermore, 
\begin{equation}\label{eq:maxepsilon}
        \max\left\{\frac{(1+\nu_1)^2}{(1-\nu_1)^4}, \frac{(1+\nu_2)^2}{(1-\nu_1)^4}, \frac{(1+\nu_1)^4}{(1-\nu_1)^4}\right\} \leq 1 + \varepsilon_1.
\end{equation}
Finally let
\begin{equation}\label{eq:N_1def}
    N_1 = \max\left\{c^{\frac{1}{1-\beta}} ,\left(\frac{1}{c^2\varepsilon_1}\right)^{\frac{1}{2\beta - 1}}, \left(\frac{64(1+\varepsilon_1)^2}{c\varepsilon_1^2}\right)^{\frac{1}{\beta}}\right\},
\end{equation}
and assume that $n\geq N_1$. Note that the condition $n \geq c^{\frac{1}{1-\beta}}$ ensures that
\begin{equation}
    p = cn^{\beta - 1} \leq 1.
\end{equation}
    We start by proving \eqref{eq:Sigmaup}. The norm of the matrix $\Sigma_2$ can be bounded using a standard inequality for singular values,
    \begin{equation}
    \begin{aligned}
        \Vert \Sigma_2 \Vert = \sigma_1(L) \leq \sigma_1(I) + \sigma_{1}(P\otimes P)&\leq 1 + \frac{\sigma_1(A)^2}{d^2(1 - R_1)^2}.
    \end{aligned}
    \end{equation}
    Note that the maximal singular value of the adjacency matrix $A$ is at most the maximal degree of the graph, so that by \eqref{eq:maxmindegree},
    \begin{equation}\label{eq:upperSigma2}
        \Vert \Sigma_2 \Vert \leq 1 + \left(\frac{1 + R_1}{1 - R_1}\right)^2 < 1 + \left(\frac{1 + \nu_1}{1 - \nu_1}\right)^2 \leq 2 + \varepsilon_1.
    \end{equation}
    Next, we prove \eqref{eq:g12up}. The norm of $g_{12}$ can be bounded by observing that
\begin{equation}\label{eq:normg12}
\begin{aligned}
    n\Vert g_{12} \Vert^2 &= nu_1^t (P\otimes P)(E-I) V_2 V_2^t (E-I)^t(P\otimes P)^t u_1 \\
    &= nu_1^t (P\otimes P)(E-I)\left(I - v_1v_1^t\right) (E-I)^t(P\otimes P)^t u_1 \\
    &= n\frac{\left(\pi\otimes\pi\right)^t (I - E) \left(\pi\otimes\pi\right)}{\Vert \pi\otimes\pi\Vert^2} - \gamma_{11}^2 \\
    &= n\frac{\sum_{i} (1 + \epsilon_i)^4}{\sum_{(i,j)} (1+\epsilon_i)^2(1+\epsilon_j)^2} - \frac{1}{n} \\
    &\leq  \left(\frac{1 + R_1}{1 - R_1}\right)^4 < \left(\frac{1 + \nu_1}{1 - \nu_1}\right)^4 \leq 1 + \varepsilon_1.
\end{aligned} 
\end{equation}
To prove \eqref{eq:invariantup}, note that
\begin{equation}
    n\Vert\pi\Vert^2 = n\sum_{i} \pi_i^2 \leq \frac{(1+R_1)^2}{(1-R_1)^2} \leq 1 + \varepsilon_1.
\end{equation}
Whereas to show \eqref{eq:G22up}, we observe that
\begin{equation}
    \begin{aligned}
       \Vert G_{22} \Vert \leq \Vert U_2\Vert \Vert (P\otimes P)(E-I)\Vert \Vert V_2\Vert \leq \Vert (P\otimes P)(E-I)\Vert.
    \end{aligned}
    \end{equation}
    Using Lemma \ref{lem:NormPerturbation}, it can be seen that,
    \begin{equation}\label{eq:normG22upper}
    \begin{aligned}
        n\Vert G_{22} \Vert^2 &\leq \frac{n(1+R_1)^2}{d^2(1-R_1)^4} + \frac{(1+R_2)^2}{(1-R_1)^4} \\
        &< \frac{n(1+\nu_1)^2}{d^2(1-\nu_1)^4} + \frac{(1+\nu_2)^2}{(1-\nu_1)^4} \\
        &\leq \left(\frac{n}{d^2} + 1\right) (1 + \varepsilon_1).
    \end{aligned}
    \end{equation}
    For $n\geq N_1$, we have in particular that $n\geq\left(\frac{1}{c^2\varepsilon_1}\right)^{\frac{1}{2\beta - 1}}$, so that
    \begin{equation}\label{eq:nd2up}
        \frac{n}{d^2} = \frac{1}{c^2}n^{1-2\beta}\leq\varepsilon_1.
    \end{equation}
    Hence, we can conclude that
    \begin{equation}
        \sqrt{n}\Vert G_{22} \Vert \leq 1 + \varepsilon_1.
    \end{equation}
        Next, we prove \eqref{eq:g21up}. Recall that 
    \begin{equation}
        n^2\Vert g_{21} \Vert^2 = n^2\widetilde{\gamma}_{11}^2 - 1.
    \end{equation}
    We can upper-bound the expression above by observing that
    \begin{equation}
    \begin{aligned}\label{eq:upperleading}
        n^2\widetilde{\gamma}_{11}^2 &\leq \frac{1}{(1-R_1)^4}\frac{1}{d^4}\sum_{(k,\ell)} \left(\sum_{i} \I\{i\sim k\}\I\{i\sim \ell\}\right)^2 \\
        &\leq \frac{1}{(1-R_1)^4} \left(2\binom{n}{2}\frac{1}{n^2}(1 + R_2)^2 + \frac{1}{np^2} (1+R_1)^2\right) \\
        &< \frac{(1+\nu_2)^2}{(1-\nu_1)^4} + \frac{1}{np^2} \frac{(1+\nu_1)^2}{(1 - \nu_1)^4} \leq (1 + \varepsilon_1)^2,
    \end{aligned}
    \end{equation}
    where the last inequality follows again from \eqref{eq:nd2up}. Finally, we prove \eqref{eq:Sigmainv}. It holds that
\begin{equation}
    \begin{aligned}
        \Vert \Sigma_2^{-2} \Vert^{-1} &= \sigma_{\text{min}}(\Sigma_2)^2 \geq (1 - \sigma_2(P\otimes P))^2, 
    \end{aligned}
\end{equation}
and
\begin{equation}
    \begin{aligned}
        \Vert \Sigma_2^{-2} \Vert^{-1} \leq (1 + \sigma_2(P\otimes P))^2.
    \end{aligned}
\end{equation}
Note that the second least singular value of $P\otimes P$ can be bounded using
\begin{equation}
    \sigma_2(P\otimes P) \leq \frac{1}{d^2(1 - R_1)^2} \sigma_2(A \otimes A) \leq \frac{(1 + R_1)}{\sqrt{d}}\sigma_2\left(\frac{A}{\sqrt{d}}\right)\leq \frac{8}{\sqrt{d}} (1 + \varepsilon_1).
\end{equation}
For $n\geq N_1$, we have in particular that $n\geq\left(\frac{64(1+\varepsilon_1)^2}{c\epsilon^2}\right)^{\frac{1}{\beta}}$, so that
    \begin{equation}
        \frac{8}{\sqrt{d}}(1+\varepsilon_1) \leq \frac{8}{\sqrt{c\left(\frac{64(1+\varepsilon_1)^2}{c\epsilon^2}\right)}} \leq \varepsilon_1.
    \end{equation}
    Therefore, the result follows.
\end{proof}
We use Lemma~\ref{lem:normestimates} to bound the terms of the rank-one approximation $\hat{t}_{\mathrm{meet}}^{\pi,(1)}$ and the error approximation in Proposition~\ref{lem:rankapprox}.
\begin{lemma}\label{lem:perturbedsigma}
Let $\beta > \frac{1}{2}$. For every $\varepsilon_2 > 0$, there exists $N_2\in\N$ and constants $\nu_1, \nu_2 > 0$ such that, for all $n \geq N_2$,
\begin{align}
\label{eq:sigmaperturbed}
&\P\left(\bigg|\frac{1}{n\sigmaPerturbed_{n^2}} -1\bigg|\leq \varepsilon_2 \mid F_{\nu_1, \nu_2}\cap F\right) = 1 
\\
\label{eq:projector}
&\P\left(|\left(\pi\otimes\pi\right)^t\vPerturbed_{n^2}\uPerturbed_{n^2}^t\mathbf{\underline{1}}_{n^2} - 1|\leq \varepsilon_2 \mid F_{\nu_1, \nu_2}\cap F\right) = 1
\\
\label{eq:error}
&\P\left(\frac{n\Vert\pi\Vert^2}{n\sigmaPerturbed_{n^2-1}}\leq \varepsilon_2 \mid F_{\nu_1}\cap F\right) = 1
\end{align}
\end{lemma}
\begin{proof} Let $\varepsilon_1 \in \left(0,\frac{1}{20}\right)$, $\nu_1, \nu_2$ be defined as in \eqref{eq:nudefs}, recall that conditioning on $F_{\nu_1,\nu_2}$ and $F$ gives the conditions that $R_1 \leq \nu_1$, $R_2 \leq \nu_2$, and $\sigma_2\left(A\right) \leq 8\sqrt{d}$. These conditions allow us to use the estimates in Lemma~\ref{lem:normestimates}. First we note that
    \begin{equation}
        n\widetilde{\gamma}_{11}^2 = 1 + n^2\Vert g_{21}\Vert^2 \geq 1.
    \end{equation}
    Therefore, the perturbed least singular value can be upper-bounded by
    \begin{equation}
        \frac{1}{n\sigmaPerturbed_{n^2}} \leq \left(n^2\widetilde{\gamma}_{11}^2 - 2n^2\frac{\Vert\Delta_{21}\Vert^2}{\delta}\right)^{-\frac{1}{2}} \leq \left(1 - 2n^2\frac{\Vert\Delta_{21}\Vert^2}{\delta}\right)^{-\frac{1}{2}},
    \end{equation}
    and lower-bounded by
    \begin{equation}
        \frac{1}{n\sigmaPerturbed_{n^2}} \geq \left(1 + n^2\Vert g_{21}\Vert^2+ 2n^2\frac{\Vert\Delta_{21}\Vert^2}{\delta}\right)^{-\frac{1}{2}}.
    \end{equation}
We will use the estimates from Lemma \ref{lem:normestimates} to conclude that for every $\varepsilon_1 \in \left(0,\frac{1}{20}\right)$, there exist constants $\nu_1, \nu_2 > 0$ and $N_1\in\R_+$ such that for $n\geq N_1$, conditioned on $F_{\nu_1,\nu_2}$ we have,
    \begin{equation}\label{eq:Delta21upper}
        \begin{aligned}
            n\Vert\Delta_{21}\Vert &\leq n\Vert g_{21}\Vert \Vert \Sigma_2 \Vert +  \Vert g_{12} \Vert + n\Vert g_{21} \Vert \Vert G_{22}\Vert \\
            &\leq 5\sqrt{\varepsilon_1}.
        \end{aligned}
    \end{equation}
    and
\begin{equation}
    n^2\Vert g_{21}\Vert^2 \leq \varepsilon_1(2+\varepsilon_1).
\end{equation}
    Furthermore, conditioned on $F_{\nu_1,\nu_2}$ and $F$ we have,
    \begin{equation}\label{eq:deltadef}
    \begin{aligned}
        \delta &= \text{sep}(B_{11}, B_{22}) - \Vert \Delta_{11} \Vert - \Vert \Delta_{22}\Vert \\
        &= \Vert \Sigma^{-2}_2\Vert^{-1} - \widetilde{\gamma}_{11}^2 - \Vert G_{22}^t\Sigma_2 + \Sigma_2G_{22} + G_{22}^tG_{22} + g_{12}g_{12}^t\Vert \\
        &\geq \Vert \Sigma^{-2}_2\Vert^{-1} - \widetilde{\gamma}_{11}^2 - 2\Vert G_{22}\Vert\Vert\Sigma_2\Vert - \Vert G_{22}\Vert^2 - \Vert g_{12}\Vert^2 \\
        &\geq 1 - 5\varepsilon_1.
    \end{aligned}
\end{equation}
Therefore,
\begin{equation}
    \frac{1}{n\sigmaPerturbed_{n^2}} \leq \left(1 - \frac{50\varepsilon_1}{1-5\varepsilon_1}\right)^{-\frac{1}{2}} \quad \text{and} \quad \frac{1}{n\sigmaPerturbed_{n^2}} \geq \left(1 + \frac{50\varepsilon_1}{1-5\varepsilon_1} + \varepsilon_1(2+\varepsilon_1)\right)^{-\frac{1}{2}}.
\end{equation}
First, note that by Cauchy-Schwarz,
    \begin{equation}
\left(\pi\otimes\pi\right)^t\vPerturbed_{n^2}\uPerturbed_{n^2}^t\mathbf{\underline{1}}_{n^2} - 1 \leq n\Vert\pi\Vert^2 - 1 \leq \varepsilon_1.
    \end{equation}
    By the matrix perturbation Theorem \ref{thm:Stewart-perturbation}, we can write
    \begin{equation}
        \vPerturbed_{(n,n)} = \left(v_{(n,n)} + V_2Q\right)\left(1 + \Vert Q\Vert^2\right)^{-\frac{1}{2}},
    \end{equation}
    where $\Vert Q\Vert \leq \frac{2\Vert\Delta_{21}\Vert}{\delta}$. Applying Theorem \ref{thm:Stewart-perturbation} to $A = LL^t$, we find that there exists a matrix $W$ with 
    \begin{equation}
        \Vert W\Vert\leq\frac{2}{\delta_W}\Vert\gamma_{11}g_{21}^t + g_{12}^t\Sigma_2 + g_{12}^tG_{22}\Vert
    \end{equation}
    such that
    \begin{equation}
        \uPerturbed_{(n,n)} = \left(u_{(n,n)} + U_2W\right)\left(1 + \Vert W\Vert^2\right)^{-\frac{1}{2}}.
    \end{equation}
    The parameter $\delta_W$ is given by
\begin{equation}
    \begin{aligned}
        \delta_W = \Vert \Sigma^{-2}_2\Vert^{-1} - \gamma_{11}^2 - \Vert g_{12}\Vert^2 - \Vert G_{22}^t\Sigma_2 + \Sigma_2G_{22} + G_{22}G_{22}^t + g_{21}g_{21}^t\Vert.
    \end{aligned}
\end{equation}
    Note that we can directly compute the unperturbed projector through
    \begin{equation}
        |\left(\pi\otimes\pi\right)^tv_{(n,n)}u_{(n,n)}^t\mathbf{\underline{1}}_{n^2}| = \frac{1}{n\Vert\pi\Vert}
        .
    \end{equation}
    Furthermore, we note that
    \begin{equation}
|\left(\pi\otimes\pi\right)^tv_{(n,n)}W^tU_2^t\mathbf{\underline{1}}_{n^2}| \leq \frac{1}{n}\Vert W\Vert\Vert U_2\Vert n \leq \Vert W\Vert,
    \end{equation}
        \begin{equation}
|\left(\pi\otimes\pi\right)^tV_2Qu_{(n,n)}^t\mathbf{\underline{1}}_{n^2}| \leq \frac{1}{\Vert\pi\Vert^2}\Vert\pi\otimes\pi\Vert\Vert V_2\Vert \Vert Q\Vert \leq \Vert Q\Vert,
    \end{equation}
    \begin{equation}
        |\left(\pi\otimes\pi\right)^tV_2QW^tU_2^t\mathbf{\underline{1}}_{n^2}| \leq n\Vert\pi\Vert^2\Vert V_2\Vert \Vert Q\Vert\Vert W\Vert\Vert U_2\Vert \leq n\Vert\pi\Vert^2 \Vert Q\Vert\Vert W\Vert.
    \end{equation}
    So that in total we obtain,
    \begin{equation}
        \left(\pi\otimes\pi\right)^t\vPerturbed_{(n,n)}\uPerturbed_{(n,n)}^t\mathbf{\underline{1}}_{n^2} \geq \frac{\left(\frac{1}{n\Vert\pi\Vert^2} - \Vert W\Vert - \Vert Q\Vert - n\Vert\pi\Vert^2\Vert W\Vert\Vert Q\Vert\right)}{(1 + \Vert Q\Vert^2)^{\frac{1}{2}}(1 + \Vert W\Vert^2)^{\frac{1}{2}}}.
    \end{equation}
    The norm of $Q$ can be bounded using \eqref{eq:Delta21upper},
    \begin{equation}
        \Vert Q\Vert \leq \frac{2\Vert\Delta_{21}\Vert}{\delta} \leq \frac{10\sqrt{\varepsilon_1}}{(1-5\varepsilon_1)n}.
    \end{equation}
    The norm of $W$ is bounded from above by
        \begin{equation}
        \Vert W\Vert \leq \frac{2\Vert\gamma_{11}g_{21}^t + g_{12}\Sigma_2 + g_{12}^tG_{22}\Vert}{\delta_W} \leq \frac{8(1+\varepsilon_1)^2}{(1-5\varepsilon_1)\sqrt{n}}.
    \end{equation}
    Note that for
    \begin{equation}
        n \geq \max\left\{\frac{10\sqrt{\varepsilon_1}}{(1-5\varepsilon_1)\varepsilon_1}, \left(\frac{8(1+\varepsilon_1)}{\varepsilon_1}\right)^2\right\},
    \end{equation}
     we have that $\Vert Q\Vert \leq \varepsilon_1$ and $\Vert W\Vert \leq \varepsilon_1$. Therefore,
    \begin{equation}
\left(\pi\otimes\pi\right)^t\vPerturbed_{(n,n)}\uPerturbed_{(n,n)}^t\mathbf{\underline{1}}_{n^2} \geq \frac{\left(\frac{1}{(1 + \varepsilon_1)} - 2\varepsilon_1 - (1+\varepsilon_1)\varepsilon_1^2\right)}{(1 + \varepsilon_1^2)}.
    \end{equation}
    Finally, we see that
    \begin{equation}
        \frac{n\Vert\pi\Vert^2}{n\sigmaPerturbed_{n^2-1}} \leq \frac{n\Vert\pi\Vert^2}{n(1 - \sigma_2(P\otimes P)))} \leq \frac{1+\varepsilon_1}{n(1-\varepsilon_1)}.
    \end{equation}
    Therefore, for $n \geq \frac{1+\varepsilon_1}{\varepsilon_1(1-\varepsilon_1)}$, 
    \begin{equation}
        \frac{n\Vert\pi\Vert^2}{n\sigmaPerturbed_{n^2-1}} \leq \varepsilon_1.
    \end{equation}
    The result follows from observing that for every $\varepsilon_2 > 0$, there exists $\varepsilon_1 \in \left(0,\frac{1}{20}\right)$ such that
    \begin{equation}
    \begin{aligned}
          &\frac{\left(\frac{1}{(1 + \varepsilon_1)} - 2\varepsilon_1 - (1+\varepsilon_1)\varepsilon_1^2\right)}{(1 + \varepsilon_1^2)} \geq 1-\varepsilon_2, \qquad  \varepsilon_1 \leq \varepsilon_2 \\
    &\left(1 - \frac{50\varepsilon_1}{1-5\varepsilon_1}\right)^{-\frac{1}{2}} \leq 1 + \varepsilon_2 \quad \text{and} \quad \left(1 + \frac{50\varepsilon_1}{1-5\varepsilon_1} + \varepsilon_1(2+\varepsilon_1)\right)^{-\frac{1}{2}} \geq 1 - \varepsilon_2.
    \end{aligned}
\end{equation}
Setting
\begin{equation}
    N_2 = \max\left\{N_1, \frac{10\sqrt{\varepsilon_1}}{(1-5\varepsilon_1)\varepsilon_1}, \left(\frac{8(1+\varepsilon_1)}{\varepsilon_1}\right)^2, \frac{1+\varepsilon_1}{\varepsilon_1(1-\varepsilon_1)} \right\},
\end{equation}
where $N_1$ is the constant in Lemma \ref{lem:normestimates}, we obtain the result.
\end{proof}
Finally, we are ready to prove Theorem \ref{thm:meetingtime}.
\begin{proof}[Proof of Theorem \ref{thm:meetingtime}]
Note that for any $\epsilon, \nu_1, \nu_2 > 0$ and $n \in \N$,
\begin{equation}
    \P\left(\bigg|\frac{1}{n}t_{\mathrm{meet}}^{\pi} - 1 \bigg| \leq \epsilon \right) \geq \P\left(\bigg|\frac{1}{n}t_{\mathrm{meet}}^{\pi} - 1 \bigg| \leq \epsilon \mid F_{\nu_1, \nu_2}\cap F\right) - \P\left(\left(F_{\nu_1, \nu_2}\cap F\right)^c\right).
\end{equation}
From Lemma \ref{lem:concentrationR} and Proposition \ref{prop:EKYYsingular}, it follows that there exists $\theta > 0$, such that for any $\nu_1, \nu_2 > 0$,
\begin{equation}
    \P\left(\left(F_{\nu_1, \nu_2}\cap F\right)^c\right) \leq 2n\exp\left(- \frac{\nu_1^2 np}{3}\right) + 2\binom{n}{2}\exp\left(- \frac{\nu_2^2 np^2}{3}\right) + \exp\left(-\theta \left(\log n\right)^{2}\right).
\end{equation}
Therefore, it remains to show that for every $\epsilon > 0$, there exist $\nu_1, \nu_2 > 0$ and $N \in \R_+$ such that for all $n\geq N$,
\begin{equation}\label{eq:proofremainder}
     \P\left(\bigg|\frac{1}{n}t_{\mathrm{meet}}^{\pi} - 1 \bigg| \leq \epsilon \mid F_{\nu_1, \nu_2}\cap F\right) = 1.
\end{equation}
We use the rank-$k$ perturbation result \ref{lem:rankapprox}, with $k = 1$, to deduce that, 
\begin{equation}
    \frac{1}{n}t_{\mathrm{meet}}^{\pi} \leq \frac{1}{n}\hat{t}_{\mathrm{meet}}^{\pi, (1)} + \frac{n\Vert\pi\Vert^2}{n\sigmaPerturbed_{n^2-1}} \leq \frac{1}{n\sigmaPerturbed_{n^2}}(\pi\otimes\pi)^t\vPerturbed_{n^2}\uPerturbed_{n^2}^t\mathbf{\underline{1}}_{n^2} + \frac{n\Vert\pi\Vert^2}{n\sigmaPerturbed_{n^2-1}}.
\end{equation}
\begin{equation}
    \frac{1}{n}t_{\mathrm{meet}}^{\pi} \geq \frac{1}{n}\hat{t}_{\mathrm{meet}}^{\pi, (1)} - \frac{n\Vert\pi\Vert^2}{n\sigmaPerturbed_{n^2-1}} \geq \frac{1}{n\sigmaPerturbed_{n^2}}(\pi\otimes\pi)^t\vPerturbed_{n^2}\uPerturbed_{n^2}^t\mathbf{\underline{1}}_{n^2} -\frac{1}{n} - \frac{n\Vert\pi\Vert^2}{n\sigmaPerturbed_{n^2 - 1}}.
\end{equation}
By Lemma \ref{lem:perturbedsigma}, for every $\varepsilon_2 > 0$, there exist constants $\nu_1, \nu_2 > 0$ and $N_2\in\R_+$ such that for all $n \geq N_2$, conditioned on $F_{\nu_1, \nu_2}$ and $F$ we have,
\begin{equation}
    \frac{1}{n}t_{\mathrm{meet}}^{\pi} \leq (1 + \varepsilon_2)^2 + \varepsilon_2.
\end{equation}
if additionally $n \geq \frac{1}{\varepsilon_2}$,
\begin{equation}
    \frac{1}{n}t_{\mathrm{meet}}^{\pi} \geq (1 - \varepsilon_2)^2 - 2\varepsilon_2.
\end{equation}
Equation \eqref{eq:proofremainder} follows from noting that for every $\epsilon > 0$, there exists $\varepsilon_2 > 0$ such that
\begin{equation}
    (1 + \varepsilon_2)^2 + \varepsilon_2 - 1 \leq \epsilon \quad \text{and} \quad 1 - (1 - \varepsilon_2)^2 + 2\varepsilon_2 \leq \epsilon.
\end{equation}
\end{proof}

\section{Discussion}
\label{sec:discussion}

We proposed an approach to estimate the meeting time of two independent random walks on graphs via SVD of the diagonally killed generator of the pair. Specifically, we viewed the diagonally killed generator as a perturbation of the (standard) generator of a pair of random walks. We devised non-asymptotic perturbation bounds on the expected meeting time via a rank-one approximation. Subsequently, we conducted non-asymptotic perturbation analysis of the rank one approximation to the diagonally killed generator. As a proof of the concept, we specialized to a sufficiently dense Erdős–Rényi random graph and obtained sharp bounds on the expected meeting time, cf.~Theorem~\ref{thm:meetingtime}. In Section~\ref{sec:dense-random}, we defined two graph characteristics $R_1$ and $R_2$ (see \eqref{eq:R_1def}, \eqref{eq:R_2def}) which we used to control the meeting-time. While our main Theorem~\ref{thm:meetingtime} is stated for sufficiently dense Erd\H{o}s–R\'enyi graphs, this dependence can be factored out to some extent. By inspecting the proof of Theorem~\ref{thm:meetingtime}, one can convince oneself that the theorem can be generalized as follows.

\begin{theorem}[Informal]
\label{thm:generalized}     
Let $(G_n)_{n\in\N}$ be a sequence of (random) graphs such that each $G_n$ has $n$ vertices. Suppose there exists a real sequence $(d_n)_{n\in\N}$ such that (w.h.p.), as $n \to \infty$, the conditions
\begin{enumerate}
    \item \label{i:r1} $\max_{i\in V_n}|\frac{\deg i}{d_n} - 1| \to 0$,
    \item \label{i:r2} $\max_{(k,\ell)\in V_n\times V_n}|\frac{n}{d_n^2}\left(\sum_{i} A_{ik}A_{i\ell}\right) - 1| \to 0$,
    \item \label{i:singular-gap} $\sigma_2(A) = O(\sqrt{d_n})$,
\end{enumerate}
hold, where $\sigma_2(A)$ is the second-largest eigenvalue of the adjacency matrix $A = A(G_n)$.
Then,
\begin{equation}
    \frac{1}{n}t_{\mathrm{meet}}^{\pi} \approx 1 \text{ (w.h.p.)}.
\end{equation}

\end{theorem}

When it comes to applying Theorem~\ref{thm:generalized} to the Erdős–Rényi random graph, it is the second condition which is the most restrictive.  Specifically, it restricts the statement of Theorem~\ref{thm:meetingtime} to sufficiently dense Erdős–Rényi random graphs ($\beta  > 0.5$). Indeed, Items \eqref{i:r1} and \eqref{i:singular-gap} hold with high probability for $p_n \gg (\log n)/n$, cf.~Lemma~\ref{lem:concentrationR} and \cite{alt2021extremal}, whereas the probability of \eqref{i:r2} is only high exactly in the regime, where the average degree grows faster than $\sqrt{n}$. We conjecture that the perturbation method can be adapted to go through under Conditions \eqref{i:r1}, \eqref{i:singular-gap} and a substantially milder condition than \eqref{i:r2}. We will explore this elsewhere.

\section{Literature}

\printbibliography[
    title={},
    heading=none
]

\appendix
\section{}
\begin{theorem}[\citet{stewart1973error}]\label{thm:Stewart-perturbation}
    Let $B, \Delta \in \mathbb{R}^{n\times n}$. Let $X = (X_1, X_2)$ be unitary with $X_1 \in \mathbb{R}^{n\times 1}$ and suppose $\mathcal{R}(X_1)$ is an invariant subspace of $B$. Let $X^tBX$ and $X^t\Delta X$ be partitioned conformally with $X$ in the forms
    \begin{equation}\label{eq:blocksA}
        X^tBX = \left[\begin{matrix}
            B_{11} & B_{12} \\
            0 & B_{22}
        \end{matrix}\right]
    \end{equation}
    and
    \begin{equation}\label{eq:blocksDelta}
        X^t\Delta X = \left[\begin{matrix}
            \Delta_{11} & \Delta_{12} \\
            \Delta_{21} & \Delta_{22}
        \end{matrix}\right]
    \end{equation}
    Let $\delta = \mathrm{sep}(B_{11}, B_{22}) - \Vert \Delta_{11} \Vert - \Vert \Delta_{22}\Vert$. If 
    \begin{equation}
        \frac{\Vert\Delta_{21}\Vert(\Vert A_{12}\Vert + \Vert \Delta_{12}\Vert)}{\delta^2} \leq \frac{1}{4},
    \end{equation}
    there exists a matrix $Q$ with $\Vert Q \Vert \leq 2\frac{\Vert \Delta_{21} \Vert}{\delta}$ such that the columns of
    \begin{equation}
        X_{1}^{'} = (X_1 + X_2Q)(I + Q^tQ)^{-1/2}
    \end{equation}
    span an invariant subspace of $B + \Delta$. Furthermore, the matrices $
    B_{11}^{'}$ and $B_{22}^{'}$ are given by
    \begin{equation}
        B_{11}^{'} = B_{11} + \Delta_{11} + (B_{12} + \Delta_{12})Q
    \end{equation}
    and
    \begin{equation}
        B_{22}^{'} = (I + QQ^t)^{1/2}(B_{22} + \Delta_{22} + Q(B_{12} + \Delta_{12}))(I + QQ^t)^{-1/2}.
    \end{equation}
    The set $\lambda(B + \Delta)$ is the disjoint union of $\lambda(B_{11}^{'})$ and $\lambda(B_{22}^{'})$.
\end{theorem}

\end{document}